\documentclass[final]{article}

\usepackage[a4paper,left=42mm,right=42mm]{geometry}

\usepackage[utf8]{inputenc}

\usepackage{amssymb,amsmath,amsthm}
\theoremstyle{plain} 
    \newtheorem{theorem}{Theorem}[section]
    \newtheorem*{theorem*}{Theorem}
    \newtheorem{corollary}[theorem]{Corollary}
    \newtheorem{lemma}[theorem]{Lemma}
    \newtheorem{proposition}[theorem]{Proposition}
\theoremstyle{definition}
    \newtheorem{definition}[theorem]{Definition}

\theoremstyle{remark}    
    
    \newtheorem{claim}{Claim}[theorem]

\usepackage{IEEEtrantools}

\newcommand{\range}{\operatorname{range}}
\newcommand{\forces}{\Vdash}
\newcommand{\satisfies}{\models}
\newcommand{\domain}{\operatorname{dom}}

\newcommand{\rank}{\operatorname{rank}}
\newcommand{\rk}{\mathrm{rk}}

\newcommand{\ZFC}{\mathrm{ZFC}}
\newcommand{\ZF}{\mathrm{ZF}}

\newcommand{\OR}{\mathrm{OR}}
    
\newcommand{\id}{\operatorname{id}}

\newcommand{\AC}{\mathrm{AC}}
\newcommand{\SRR}{\mathrm{SRR}}
\newcommand{\RR}{\mathrm{RR}}

\newcommand{\GCH}{\mathrm{GCH}}

\newcommand{\LST}{\mathrm{LST}}
\newcommand{\FORM}{\mathrm{FORM}}
\newcommand{\length}{\operatorname{length}}

\newcommand{\relsubbar}[1]{\mathrel{\underline{#1}}}
\newcommand{\subbar}[1]{\underline{#1}}

\usepackage{comment}

\usepackage{soul}

\usepackage{enumitem}
\setlist[enumerate]{label={\upshape(\roman*)},noitemsep}

\usepackage{tabularx}
\newcolumntype{C}{>{\centering\arraybackslash$}X<{$}}
\newcolumntype{M}{>{\centering\arraybackslash$}p{.35cm}<{$}}

\usepackage{todonotes}

\usepackage{biblatex}
\renewbibmacro{in:}{}
\usepackage{graphicx}
\usepackage{caption}

\usepackage{placeins} 

\makeatletter
\DeclareFontFamily{OMX}{MnSymbolE}{}
\DeclareSymbolFont{MnLargeSymbols}{OMX}{MnSymbolE}{m}{n}
\SetSymbolFont{MnLargeSymbols}{bold}{OMX}{MnSymbolE}{b}{n}
\DeclareFontShape{OMX}{MnSymbolE}{m}{n}{
    <-6>  MnSymbolE5
   <6-7>  MnSymbolE6
   <7-8>  MnSymbolE7
   <8-9>  MnSymbolE8
   <9-10> MnSymbolE9
  <10-12> MnSymbolE10
  <12->   MnSymbolE12
}{}
\DeclareFontShape{OMX}{MnSymbolE}{b}{n}{
    <-6>  MnSymbolE-Bold5
   <6-7>  MnSymbolE-Bold6
   <7-8>  MnSymbolE-Bold7
   <8-9>  MnSymbolE-Bold8
   <9-10> MnSymbolE-Bold9
  <10-12> MnSymbolE-Bold10
  <12->   MnSymbolE-Bold12
}{}
\DeclareMathDelimiter{\ulcorner}
    {\mathopen}{MnLargeSymbols}{'036}{MnLargeSymbols}{'036}
\DeclareMathDelimiter{\urcorner}
    {\mathclose}{MnLargeSymbols}{'043}{MnLargeSymbols}{'043}
\makeatother

\title{Strong Rigidity and Elementary Embeddings}
\author{Marwan Salam Mohammd
\\
    \texttt{marwan.mizuri@gmail.com}
}

\begin{document}

\maketitle

\begin{abstract}
    We present a method for producing elementary embeddings from homomorphisms. This method is utilized in the study of the ``strongly rigid relation principle'' as defined by Hamkins and Palumbo in their paper ``The Rigid Relation Principle, a New Weak Choice Principle.'' We establish that the strongly rigid relation principle is also a weak choice principle that is independent of ZF. Finally, we characterize proto Berkeley cardinals in terms of a strong failure of the strongly rigid relation principle.
\end{abstract}



\section{Introduction}

In this paper, we will be dealing with structures endorsed with a single binary relation. Such structures are intuitively easier to visualize as directed graphs. A \emph{directed graph} $(G,E)$ is an ordered pair consisting of a set $G$ of \emph{vertices} and a set $E\subset G\times G$ of \emph{arrows}. For the sake of being concise, we will drop the word ``directed'' in this definition. This should cause no confusion as we will be dealing only with directed graphs. When talking about some graph $(G,E),$ we might use the binary relation notation to indicate arrows, that is, we might write $x \mathrel{E} y$ to indicate $(x,y)\in E.$ 

Recall that a map $h:G_1\rightarrow G_2$ is called a \emph{homomorphism} from the graph $(G_1,E_1)$ to the graph $(G_2,E_2)$ iff $(x,y)\in E_1 \implies (h(x),h(y))\in E_2.$ The fact that $h$ is a homomorphism from $(G_1,E_1)$ to $(G_2,E_2)$ will be indicated by writing $h:(G_1,E_1) \rightarrow (G_2,E_2).$ An \emph{isomorphism} is a bijective map $h:G_1\rightarrow G_2$ such that both $h$ and $h^{-1}$ are homomorphisms. A homomorphism from a graph to itself is called an \emph{endomorphism} and an isomorphism from a graph to itself is called an \emph{automorphism}.

\begin{definition}
    A binary relation $E$ on a set $G$ is said to be \emph{rigid} iff the graph $(G,E)$ has no nontrivial automorphisms. If, furthermore, $(G,E)$ has no nontrivial endomorphisms, then $E$ is said to be \emph{strongly rigid.} If $E$ is a (strongly) rigid relation on $E,$ then we say that the graph $(G,E)$ is \emph{(strongly) rigid}.
\end{definition}

\begin{definition}[\cite{hamkins}]
    The \emph{rigid relation principle}, $\RR,$ asserts that every set $G$ has a rigid relation $E.$ The \emph{strongly rigid relation principle}, $\SRR,$ asserts that every set $G$ has a strongly rigid relation $E.$
\end{definition}

$\SRR$ implies $\RR$ by definition. \citeauthor{vopenka} \cite{vopenka} prove, using the Axiom of Choice $(\AC),$ that every set has a strongly rigid relation.\footnote{In graph theory literature, what we call ``strongly rigid'' (following \cite{hamkins}) is called just ``rigid.''} Thus, $\AC$ implies $\SRR,$ and hence also $\RR.$ Joel David Hamkins and Justin Palumbo \cite{hamkins} prove that $\RR$ does not imply $\AC$ and also establish the independence of $\RR$ from $\ZF.$ We therefore have the following theorem:

\begin{theorem}[\cite{vopenka} and \cite{hamkins}]\label{thm:rr_principle}
    $\RR$ is independent from $\ZF,$ follows from $\AC,$ but is not equivalent to $\AC.$
\end{theorem}

We will prove in this paper that the same is true for $\SRR$ as well. That is,

\begin{theorem}
    $\SRR$ is independent from $\ZF,$ follows from $\AC,$ but is not equivalent to $\AC.$
\end{theorem}

The independence of $\SRR$ from $\ZF$ already follows from independence of $\RR$ and the fact that $\neg\RR\implies \neg\SRR.$ Also, \cite{vopenka} already proves that $\SRR$ follows from $\AC,$ as mentioned above. So, we only need to prove that $\SRR$ does not imply $\AC.$ We do so by combining ideas from \cite{hamkins} and a method for producing elementary embeddings from homomorphisms. First, we will prove the following in section 3:

\theoremstyle{plain}
\newtheorem*{thm:reals_case}{Theorem~\ref{thm:reals_case}}
\begin{thm:reals_case}
    If there exists a set $G\subset \mathbb{R}\times \OR$ for which there is no strongly rigid relation, then for some ordinal $\alpha$ there exists a nontrivial elementary embedding $j:V_\alpha^L\rightarrow V_\alpha^L.$
\end{thm:reals_case}

Using that, in section 4, we prove that $\SRR$ does not imply $\AC$ by building a model for $\ZF+\neg\AC+\SRR.$

\theoremstyle{plain}
\newtheorem*{thm:neg_ac_plus_srr}{Theorem~\ref{thm:neg_ac_plus_srr}}
\begin{thm:neg_ac_plus_srr}
    If $\ZF$ is consistent, then so is $\ZF+\neg\AC+\SRR.$
\end{thm:neg_ac_plus_srr}

In the last section, we will take a look at Proto Berkeley cardinals, a large cardinal notion inconsistent with $\AC$ introduced in \cite{lcbc}. Their definition is given in terms of the existence of elementary embeddings. 

\begin{definition}[\cite{lcbc}]
    A cardinal $\delta$ is \emph{proto Berkeley} iff for every transitive set $M$ such that $\delta\in M,$ there is a nontrivial elementary embedding $j:M\rightarrow M$ with critical point strictly below $\delta.$
\end{definition}

We will use our method for getting elementary embeddings from homomorphisms to characterize these cardinals in terms of the existence of homomorphisms rather than elementary embeddings. This is useful since the notion of a homomorphism is simpler and more ``elementary'' than the notion of elementary embeddings, and one would naturally want their (large cardinal) axioms to be as simple as possible. Recall that the \emph{Hartog's number} of a set $X,$ denoted by $\aleph(X),$ is the least ordinal $\alpha$ such that there is no injection from $\alpha$ into $X.$

\theoremstyle{plain}
\newtheorem*{thm:proto_berkeley_cardinals}{Theorem~\ref{thm:proto_berkeley_cardinals}}
\begin{thm:proto_berkeley_cardinals}
    A cardinal $\delta$ is proto Berkeley iff for any graph $(G,E)$ such that $\aleph(G)>\delta$ and any injection $f:\delta\rightarrow G,$ there is an endomorphism $h:(G,E)\rightarrow(G,E)$ such that $h\vert_{\range{f}}\neq \id.$
\end{thm:proto_berkeley_cardinals}


\section{Countable Sets and Ordinal Numbers}\label{sec:ordinals_case}

In this section, we first make the observation that every at most countable set has a strongly rigid relation, and then we present our main method for producing elementary embeddings from homomorphisms. In particular, we will prove that if for some ordinal $\beta$ there is no strongly rigid relation (so that there are plenty of homomorphisms), then there exists an elementary embedding $j:V_\alpha^L\rightarrow V_\alpha^L$ for some ordinal $\alpha.$ $V_\alpha^L$ here is simply $V_\alpha$ as computed in $L.$

By a \emph{countable} set we mean a set that is in bijection with $\omega.$ A set that is either countable or finite is said to be \emph{at most countable}. An \emph{uncountable} set is one that is not at most countable. Although in the context of $\ZF$ it might not be true that every set has a cardinality, we can still define the \emph{cardinality} of a set $x$ to be the least ordinal that is in bijection with $x,$ if it exists, denoting it by $|x|.$

We will also need some basic graph theoretic terminology before we can continue. Let $(G,E)$ be some graph. A \emph{loop} is an arrow $(u,v)\in E$ such that $u=v.$ A \emph{$n$-cycle} is a sequence $\langle v_0,\ldots,v_{n}\rangle$ of vertices such that $(v_i,v_{i+1})\in E,$ for all $i\in n,$ and $v_0=v_n.$ For any $n\in \omega$ and vertex $v,$ the \emph{outdegree} of $v$ is said to be $n$ iff there are exactly $n$ arrows outgoing from $v,$ that is, $|\{(v,w)\in E\mid w\in G\}|=n.$ Given $H\subset G,$ the \emph{subgraph induced by $H$} is the graph $(H,F)$ where $F=\{(u,v)\in E\mid u,v\in H\}.$

\begin{proposition}\label{prop:countable_case}
    Every at most countable set has a strongly rigid relation.
\end{proposition}

\begin{proof}
    Let $G$ be a set that is at most countable. If $G$ is countable, then let $E$ be such that $(G,E)$ is the graph in Figure~\ref{fig:countable_case}. Fix an endomorphism $h:(G,E)\rightarrow (G,E).$ First, notice that $u_0$ has outdegree $2$ due to the two arrows connecting it with $u_1$ and $u_2.$ Since $u_1$ and $u_2$ are connected by an arrow, and since $(G,E)$ is free of loops $h(u_1)\neq h(u_2).$ Thus, $h(u_0)$ must have outdegree at least $2.$ But, $u_0$ is the only vertex with outdegree at least $2,$ and therefore $h(u_0)=u_0.$ Now, $h$ can either fix both $u_1$ and $u_2$ or swap them. However, the arrow $(u_1,u_2)$ ensures that swapping them is not possible, so that $h$ fixes $u_1$ and $u_2$ too. It is now easy to see that $h$ must fix every other vertex as well, meaning that $h=\id.$

    If $G$ is finite, say $G=\{u_0,\ldots,u_n\},$ then simply take the induced subgraph of the graph in Figure~\ref{fig:countable_case}.
\end{proof}

\begin{figure}[htb!]
    \centering
    \includegraphics[page=6]{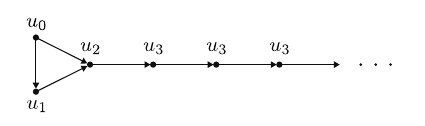}
    \caption{A strongly rigid graph on a countable set}
    \label{fig:countable_case}
\end{figure}



The language of set theory, $\LST,$ is defined in $V$ to be the first order language with the binary relation symbol $\in.$ Denote by $\FORM$ the set of all formulas of $\LST.$ To distinguish metatheoretic formulas from those in $V,$ whenever a formula appears in this paper, we will mention that it belongs to the set $\FORM$ precisely when the formula is in $V.$

\begin{theorem}\label{thm:ordinal_case}
    If there exists an ordinal $\beta$ for which there is no strongly rigid relation, then for some ordinal $\alpha$ there exists a nontrivial elementary embedding $j:V_\alpha^L\rightarrow V_\alpha^L.$
\end{theorem}

\begin{proof}
    Let $\kappa=|\beta|.$ Since $\kappa$ is a cardinal in $V,$ it is also a cardinal in $L.$ The fact that $\GCH$ holds in $L$ implies that for every cardinal $\lambda$ in $L$ there is some ordinal $\alpha$ such that $|V_\alpha^L|^L=\lambda.$ Fix $\alpha$ such that $|V_\alpha^L|^L=\kappa.$ Our argument would be easier if $\alpha$ were a limit ordinal so that $V_\alpha^L$ is closed under taking ordered pairs and finite sequences. But, we do not know this for sure, so let $N\supset V_\alpha^L\cup \{V_\alpha^L\}$ be the smallest set closed under taking finite subsets, i.e., $\{x_1,\ldots,x_n\}\subset N\implies \{x_1,\ldots,x_n\}\in N,$ for all $n\in\omega.$ Notice that $N\in L$ and has cardinality $\kappa$ in $L,$ and since $\kappa$ is a cardinal in $V$ too, we have that $|N|=\kappa.$ Furthermore, $N$ is transitive and closed under the operations:
    \begin{enumerate}[label=N\arabic*]
        \item $x,y \mapsto (x,y).$
        \item $x_1,\ldots, x_n \mapsto \langle x_1,\ldots, x_n\rangle,$ for all $n\in\omega.$
    \end{enumerate}


    Let $N^c$ be a copy of $N.$ Define $<_L^c\subset N^c\times N^c$ by $x^c<_L^c y^c \iff x<_L y.$ Consider the graph $(G,E)$ in Figure~\ref{fig:L_case}, where
    \begin{enumerate}[label=G\arabic*]
        \item $(w_0,x),(x,x^c),(w_1,x^c)\in E,$ for every $x\in N.$
        \item $(x,y)\in E$ iff $x\in y,$ for every $x,y\in N.$
        \item $(x^c,y^c)\in E$ iff $x^c<_L^c y^c,$ for every $x^c,y^c\in N^c.$
        \item $(p_n,x^c)\in E$ iff $|x|=n,$ for every $n\in\omega$ and every $x^c\in N^c.$
        \item For every $n\in \omega$ and every $x^c\in N^c,$ $(q_n,x^c)\in E$ iff one of the followings hold:
        \begin{enumerate}[label=(\alph*)]
            \item $x\in V_\alpha^L,$ $n=\ulcorner\psi(a)\urcorner$ for some $\psi(a)\in\FORM,$ and $V_\alpha^L\satisfies \psi[x].$
            \item $x=\langle x_1,\ldots,x_m\rangle$ for some $m>1,$ $x_1,\ldots,x_m\in V_\alpha^L,$ $n=\ulcorner \psi(a_1,\ldots,a_m)\urcorner$ for some $\psi(a_1,\ldots,a_m)\in \FORM,$ and $V_\alpha^L\satisfies \psi[x_1,\ldots,x_m].$
        \end{enumerate}
    \end{enumerate}
    \begin{figure}[htb!]
        \centering
        \includegraphics[page=3]{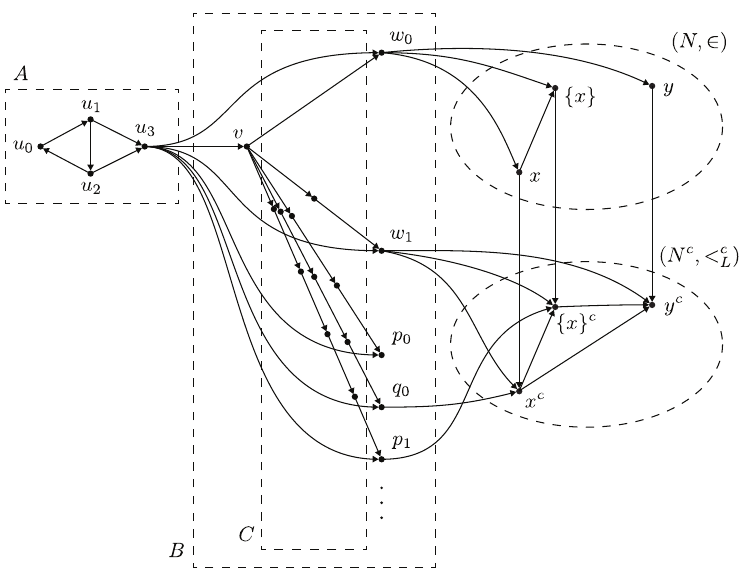}
        \caption{}
        \label{fig:L_case}
    \end{figure}
    The cardinality of $G$ is $|A|+|B|+|N|+|N^c|=4+\omega+\kappa+\kappa=\kappa.$ Thus, the graph $(G,E)$ has a nontrivial endomorphism $h:(G,E)\rightarrow (G,E).$ We will show that the restriction $h\vert_{V_\alpha^L}$ is a nontrivial elementary embedding of $V_\alpha^L$ into itself.

    We start by showing that $h$ fixes every vertex in $A.$ Observe that the $3$-cycle $c=(u_0,u_1,u_2,u_3)$ is the only cycle in the entire graph: None of the arrows outside of those in $c, N,$ and $N^c$ can belong to a cycle because they only flow from left to right, while inside of each of $N$ and $N^c,$ the wellfoundedness of $\in$ and $<_L$ prohibit cycles. As such, $h$ must send $c$ to itself. The two arrows $(u_1,u_3)$ and $(u_2,u_3)$ ensure that $c$ is not rotated by $h,$ and it is now straightforward to see that $h$ fixes everything in $A.$

    Next, we show that every vertex in $B$ is fixed. Notice that the vertices in $B-C$ are unique in that they are precisely those that have an incoming arrow from $u_3.$ As $h$ fixes $u_3,$ we have that $B-C$ is closed under $h.$ Consider now the subgraphs of the form shown in Figure~\ref{fig:length_considerations}. Since the end vertices of such subgraphs are sent to the end vertices of similar subgraphs, and since $h$ must maintain lengths of such graphs, we conclude that every vertex in $B$ must be fixed.

    \begin{figure}[htb!]
        \centering
        \includegraphics[page=7]{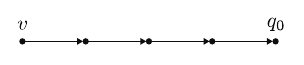}
        \caption{}
        \label{fig:length_considerations}
    \end{figure}

    Similar to how $B-C$ was closed under $h,$ the two sets $N$ and $N^c$ are also closed under $h.$ Furthermore, because $(x,y^c)\in E$ iff $x=y$ by G1, we must have $h(x)^c=h(x^c),$ for all $x\in N.$ Thus, $h$ behaves the same in $N$ and $N^c.$ Now, it is easy to see that $h\vert_{N}$ is injective because $h\vert_{N^c}$ is:
    \begin{align*}
        x^c\neq y^c &\implies (x^c,y^c)\in E \ \vee\ (y^c, x^c) \in E\\
        &\implies (h(x^c),h(y^c))\in E\ \vee\ (h(y^c), h(x^c)) \in E\\ &\implies h(x^c)\neq h(y^c).
    \end{align*}

    Let us show that, for all $n\in\omega$ and all distinct $x_1,\ldots,x_n\in N,$
    \begin{equation}\label{eq:h_respects_finite_sets}
        h(\{x_1,\ldots,x_n\})=\{h(x_1),\ldots,h(x_n)\}.
    \end{equation}
    Denote by $S$ the set $\{x_1,\ldots,x_n\}.$ We have $(x_i,S)\in E,$ for all $i,$ by G2. By endomorphism of $h,$ we have $(h(x_i),h(S))\in E,$ for all $i.$ This means that $h(x_i)\in h(S),$ for all $i,$ by G2. We already know that $h$ is injective on $N,$ so we deduce that $h(S)$ has at least the $n$ distinct elements $h(x_1),\ldots, h(x_n).$ We will be done if we can show that the cardinality of $h(S)$ is $n.$ By G4, $(p_n,S^c)\in E,$ and by applying $h$ to this we get $(h(p_n),h(S^c))\in E.$ $p_n$ is fixed by $h$ and we saw in the previous paragraph that $h(S^c)=h(S)^c,$ therefore, $(p_n,h(S)^c)\in E.$ Again by G4, $|h(S)|=n.$

    Using \eqref{eq:h_respects_finite_sets} above, we can easily show that, for all $(x,y)\in N,$ $h((x,y))=(h(x),h(y)):$ 
    \begin{align*}
        h((x,y)) 
        &= h(\{\{x\},\{x,y\}\}) \\
        &= \{h(\{x\}),h(\{x,y\})\} \\
        &= \{\{h(x)\},\{h(x),h(y)\}\} \\
        &= (h(x),h(y)).
    \end{align*}
    Observe that, by G5(a) and the fact that all $m\in\omega$ are uniquely definable inside of $V_\alpha^L,$ we must have $h(m)=m$ for all $m\in\omega.$ We can now also show that, for all $\langle x_1,\ldots,x_n\rangle\in N,$ $h(\langle x_1,\ldots,x_n\rangle)= \langle h(x_1),\ldots,h(x_n)\rangle:$
    \begin{align*}
        h(\langle x_1,\ldots,x_n\rangle) 
        &= h(\{(0,x_1),\ldots,(n-1,x_n)\}) \\
        &= \{h((0,x_1)),\ldots,h((n-1,x_n))\} \\
        &= \{(h(0),h(x_1)),\ldots,(h(n-1),h(x_n))\} \\
        &= \{(0,h(x_1)),\ldots,(n-1,h(x_n))\} \\
        &= \langle h(x_1),\ldots,h(x_n)\rangle.
    \end{align*}

    There are two last facts that we will need before we present our final argument. The first one is the fact that $x\in V_\alpha^L \implies h(x)\in V_\alpha^L.$ We show this by proving that $h$ fixes $V_\alpha^L,$ so that $x\in V_\alpha^L \implies h(x)\in h(V_\alpha^L)=V_\alpha^L.$ Observe that every set in $N-V_\alpha^L$ is finite except for $V_\alpha^L,$ which is infinite because $|V_\alpha^L|=\kappa>\omega$ by Proposition~\ref{prop:countable_case}. Also, $h(V_\alpha^L)$ must be an infinite set too by injectivity of $h$ on $(N,\in),$ so we will be done if we can show that $h(V_\alpha^L)\notin V_\alpha^L.$ But, $h(V_\alpha^L)\in V_\alpha^L$ is clearly not possible because repeated applications of $h$ to this relation can give an infinitely descending chain $V_\alpha^L\ni h(V_\alpha^L)\ni h^2(V_\alpha^L)\ni\ldots.$

    The other fact we need is that indeed $h$ is nontrivial on $V_\alpha^L.$ Suppose that this is not the case. We have already shown that $h$ fixes $V_\alpha^L$ as well as everything outside of $N$ and $N^c.$ A simple inductive argument using the fact that $h(\{x_1,\ldots,x_n\})=\{h(x_1),\ldots, h(x_n)\},$ for any $n\in\omega,$ will easily demonstrate that $h$ must fix everything in $N-(V_\alpha^L\cup \{V_\alpha^L\})$ too. This contradicts the fact that $h:(G,E)\rightarrow (G,E)$ is nontrivial. 
    
    We are now ready to conclude our proof. Fix a formula $\psi(a_1,\ldots,a_m)\in \FORM$ for some $m>1,$ the case where $\psi$ is a single variable formula is the same. We have, for any $x_1,\ldots,x_m\in V_\alpha^L,$
    \begin{align*}
        V_\alpha^L \satisfies \psi[x_1,\ldots,x_m]
        \iff& (q_{\ulcorner \psi(a_1,\ldots,a_m)\urcorner},\langle x_1,\ldots, x_m\rangle^c)\in E \\
        \implies& (h(q_{\ulcorner \psi(a_1,\ldots,a_m)\urcorner}),h(\langle x_1,\ldots, x_m\rangle^c))\in E \\
        \iff& (q_{\ulcorner \psi(a_1,\ldots,a_m)\urcorner},\langle h(x_1),\ldots, h(x_m)\rangle^c)\in E \\
        \iff& V_\alpha^L\satisfies \psi(h(x_1),\ldots, h(x_m)).
    \end{align*}
    By using $\neg \psi$ in place of $\psi$ in the above displayed argument, we can also reverse that single one way implication in the second line. Thus, for any $x_1,\ldots,x_m\in V_\alpha^L,$
    \[
    V_\alpha^L\satisfies \psi(x_1,\ldots,x_m)\iff V_\alpha^L\satisfies \psi(h(x_1),\ldots, h(x_m)).\qedhere
    \]
\end{proof}


\section{Subsets of $\mathbb{R}\times\OR$}

In this section, we generalize the theorem from the previous section to sets $G\subset \mathbb{R}\times \OR.$ We take $\mathbb{R}$ to be the set of all binary sequences of length $\omega,$ i.e., $\mathbb{R}=2^\omega.$ The relation ${<}\subset \mathbb{R}\times\mathbb{R}$ is the usual ``less than'' relation of real numbers. We will need the following lemma first:

\begin{lemma}\label{lem:reals_case_no_countable_subset}
    Every set $G\subset \mathbb{R}\times \OR$ with no countable subset has a strongly rigid relation.
\end{lemma}

\begin{proof}
    Define the relation $E$ on $G$ by setting $(r_1,\xi_1)\mathrel{E}(r_2,\xi_2)$ iff $\xi_1<\xi_2$ or $\xi_1=\xi_2\ \land\ r_1<r_2.$ Suppose $h:(G,E)\rightarrow (G,E)$ is a nontrivial endomorphism, and let $(r,\xi)\in G$ be such that $h((r,\xi))\neq (r,\xi).$ Without loss of generality, we may assume $(r,\xi)\mathrel{E} h((r,\xi)).$ Repeated applications of $h$ to this relation will give $(r,\xi)\mathrel{E} h((r,\xi))\mathrel{E}h^2((r,\xi))\mathrel{E}\ldots.$ But, this chain forms a countable subset of $G,$ contrary to the assumption that $G$ has no countable subsets.
\end{proof}

Let $\pi_2:V\rightarrow V$ be the operation defined by $(x_1,x_2)\mapsto x_2.$ For any finite sequence $s=\langle x_0,\ldots,x_{n-1}\rangle\in V^{<\omega},$ define $\length(s)=\domain(s)=n.$

An ordinal is said to be \emph{odd} iff it can be written in the form $\alpha+2n+1$ for some limit ordinal $\alpha$ and some $n\in\omega.$ An ordinal is \emph{even} iff it is not odd. Denote by $\OR^{\mathrm{o}}$ and $\OR^{\mathrm{e}}$ the classes of odd and even ordinals, respectively. 

\begin{theorem}\label{thm:reals_case}
    If there exists a set $G\subset \mathbb{R}\times \OR$ for which there is no strongly rigid relation, then for some ordinal $\alpha$ there exists a nontrivial elementary embedding $j:V_\alpha^L\rightarrow V_\alpha^L.$
\end{theorem}

\begin{proof}
    Take such a set $G.$ By Proposition~\ref{prop:countable_case}, $G$ must be uncountable, and by Lemma~\ref{lem:reals_case_no_countable_subset}, it must have a countable subset. Fix $H\subset G$ that is countable, and let $K=G-H.$ By simply renaming ordinals, we may assume that $\pi_2[K]=\kappa,$ where $\kappa$ is the cardinality of $\pi_2[K].$

    \begin{claim}\label{claim:thm_reals_case}
        $\kappa\geq\omega.$
    \end{claim}

    \begin{proof}[Proof of claim]
        We will show that if $\kappa<\omega,$ then $G$ would have a strongly rigid relation, contrary to the choice of $G.$ So, assume that $\kappa<\omega.$ Notice that $K$ cannot be empty by Proposition~\ref{prop:countable_case}, hence $0<\kappa<\omega.$ Fix $n\in\omega$ such that $\kappa=n+1.$ Partition $K$ by defining, for each $i\in n+1,$ $K_i=\{(r,\xi)\in K \mid \xi=i\}.$
        Consider the graph $(G,E)$ shown in Figure~\ref{fig:reals_case_countable_subset_finite_kappa}, where 
        \begin{enumerate}[label=G\arabic*]
            \item For every $i\in n+1:$ $(w_i,(r,i))\in E,$ for all $(r,i)\in K_i.$
            \item $(e_s,(r,\xi))\in E$ iff $r\vert_{\length(s)}=s,$ for all $s\in 2^{<\omega}$ and all $(r,\xi)\in K.$
        \end{enumerate}
        
        \begin{figure}[htb!]
            \centering
            \includegraphics[page=5]{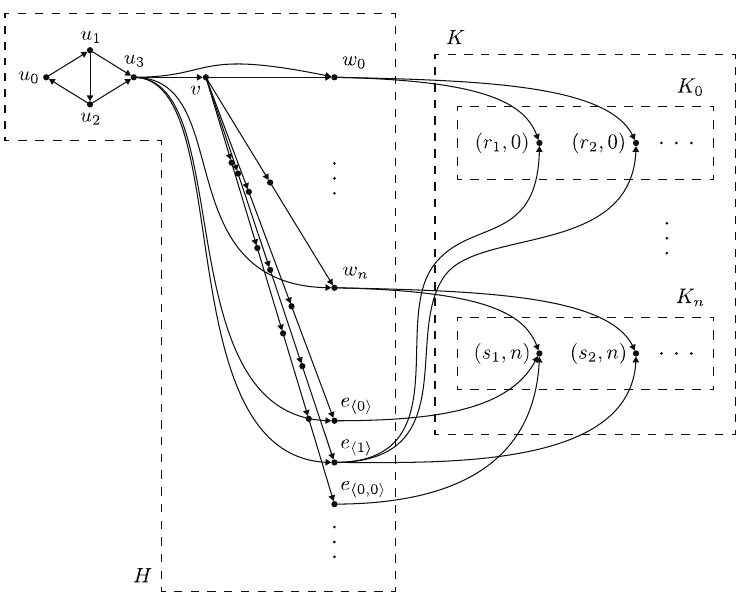}
            \caption{}
            \label{fig:reals_case_countable_subset_finite_kappa}
        \end{figure}

        Fix an endomorphism $h:(G,E)\rightarrow (G,E).$ Arguing just as in the proof of Theorem~\ref{thm:ordinal_case}, all the vertices in $H$ must be fixed by $h$ and, moreover, each of the $K_i$s is closed under $h.$ Suppose, working towards a contradiction, that for some fixed $i\in n+1,$ $h((r_1,i))=(r_2,i)$ for distinct $(r_1,i),(r_2,i)\in K_i.$ Fix $k\in\omega$ such that $r_1\vert_{k}\neq r_2\vert_{k}.$ By G2, we must have $(e_{r_1\vert_{k}},(r_1,i))\in E.$ Applying $h,$ we get $(h(e_{r_1\vert_{k}}),h((r_1,i))) = (e_{r_1\vert_{k}},(r_2,i))\in E.$ By G2, this means that $r_2\vert_{k}=r_1\vert_{k},$ a contradiction. Thus, $h$ must fix every vertex in each of the $K_i$s as well, so that $h=\id.$ As $h$ was arbitrary, we conclude that $E$ is a strongly rigid relation on $G.$
    \end{proof}

    Continuing with the proof of the theorem, we will use the following notation: Given any two sets $X,Y$ of vertices, instead of saying, for all $a\in X$ and all $b\in Y,$ $(a,b)\in E,$ we will simply say $(X,Y)\in E^*.$ In a figure, while a small black node represents a vertex, sets of vertices will be represented by a slightly bigger node that is white with a black circumference (see Figure~\ref{fig:reals_case_countable_subset_infinite_kappa} for an example).
    The fact that $(X,Y)\in E^*$ will be represented by a dotted arrow from the node for $X$ (or $a$ if $X=\{a\}$) to the node for $Y$ (or $b$ if $Y=\{b\}$). Unlike a dotted arrow, an arrow with dots and dashes only indicates that there are some, but not necessarily all possible, arrows from the vertices at its tail to the vertices at its head.

    Let $\kappa^{\mathrm{e}}=\kappa\cap\OR^{\mathrm{e}}$ and $\kappa^{\mathrm{o}}=\kappa\cap\OR^{\mathrm{o}},$ and observe that $\kappa=|\kappa^{\mathrm{e}}|=|\kappa^{\mathrm{o}}|.$ Partition $K$ into the two sets $K^{\mathrm{e}}=\{(r,\xi)\in K\mid \pi_2(\xi)\in \kappa^{\mathrm{e}}\}$ and $K^{\mathrm{o}}=\{(r,\xi)\in K\mid \pi_2(\xi)\in \kappa^{\mathrm{o}}\}.$ Similar to the proof of Theorem~\ref{thm:ordinal_case}, we can fix $V_\alpha^L$ such that $|V_\alpha^L|=\kappa$ and take $N,$ and its copy $N^c,$ such that $N\supset (V_\alpha^L\cup\{V_\alpha^L\})$ is the smallest set closed under taking finite subsets. As $|N|=\kappa=|\kappa^{\mathrm{e}}|,$ we can take a bijection $f:\kappa^{\mathrm{e}}\rightarrow N,$ and use it to replace each $(r,\xi)\in K^{\mathrm{e}}$ with $(r,f(\xi)).$ Thus, we may assume that $\pi_2[K^{\mathrm{e}}]=N.$ Similarly, we may assume that $\pi_2[K^{\mathrm{o}}]=N^c.$
    
    We partition each of $K^{\mathrm{e}}$ and $K^{\mathrm{o}}$ by the equivalence $(r,x)\sim (s,y)$ iff $x=y,$ and denote the class of $(r,x)$ by $[x].$ Let $\subbar{N}=\{[x]\mid x\in N\}$ and $\subbar{N^c}=\{[x^c]\mid x^c\in N^c\}.$ Define the relation $\subbar{\in}\subset \subbar{N}\times \subbar{N}$ by $[x]\relsubbar{\in}[y] \iff x\in y,$ for all $x,y\in N.$ Similarly, define $\subbar{<_L^c}\subset \subbar{N^c}\times \subbar{N^c}$ and $<_L^c\subset N^c\times N^c$ by $[x^c]\relsubbar{<_L^c} [y^c] \iff x^c<_L^c y^c \iff x<_L y,$ for all $x^c,y^c\in N^c.$

    Consider the graph $(G,E)$ in Figure~\ref{fig:reals_case_countable_subset_infinite_kappa}, where
    \begin{enumerate}[label=G\arabic*]
        \item $(\{w_0\},[x]),([x],[x^c]),(\{w_1\},[x^c])\in E^*,$ for every $x\in N.$
        \item $([x],[y])\in E^*$ iff $[x]\relsubbar{\in}[y],$ for every $[x],[y]\in\subbar{N}.$
        \item $([x^c],[y^c])\in E^*$ iff $[x^c]\relsubbar{<_L^c}[y^c],$ for every $[x^c],[y^c]\in \subbar{N^c}.$
        \item $(\{p_n\},[x^c])\in E^*$ iff $|x|=n,$ for every $n\in\omega$ and every $[x^c]\in \subbar{N^c}.$
        \item For every $n\in \omega$ and every $[x^c]\in \subbar{N^c},$ $(\{q_n\},[x^c])\in E^*$ iff one of the followings hold:
        \begin{enumerate}[label=(\alph*)]
            \item $x\in V_\alpha^L,$ $n=\ulcorner\psi(a)\urcorner$ for some $\psi(a)\in\FORM,$ and $V_\alpha^L\satisfies \psi[x].$
            \item $x=\langle x_1,\ldots,x_m\rangle$ for some $m>1,$ $x_1,\ldots,x_m\in V_\alpha^L,$ $n=\ulcorner \psi(a_1,\ldots,a_m)\urcorner$ for some $\psi(a_1,\ldots,a_m)\in \FORM,$ and $V_\alpha^L\satisfies \psi[x_1,\ldots,x_m].$
        \end{enumerate}
        \item $(e_s,(r,\xi))\in E$ iff $r\vert_{\length(s)}=s,$ for all $s\in 2^{<\omega}$ and all $(r,\xi)\in K.$
    \end{enumerate}
    Figure~\ref{fig:dotted_lines_expanded} depicts what the dotted arrows represent by showing the subgraph of $(G,E)$ induced by $\{w_0\}\cup [x] \cup[x^c].$

        \begin{figure}[htb!]
            \centering
            \includegraphics[page=1]{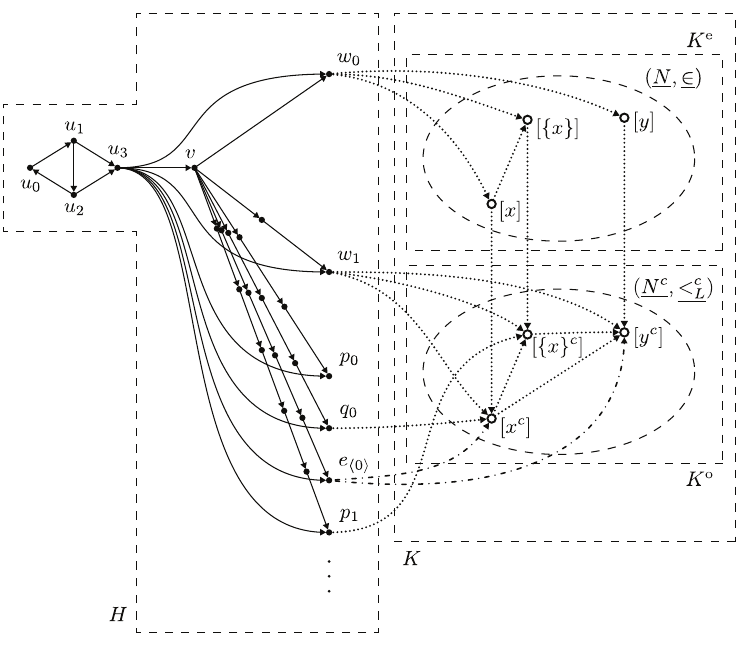}
            \caption{}
            \label{fig:reals_case_countable_subset_infinite_kappa}
        \end{figure}

        \begin{figure}[htb!]
            \centering
            \includegraphics[page=2]{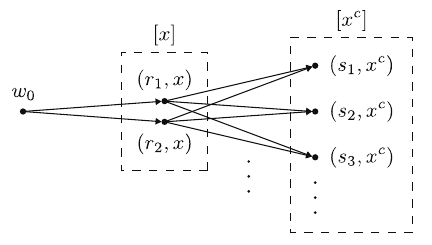}
            \caption{}
            \label{fig:dotted_lines_expanded}
        \end{figure}

    Since $G$ has no strongly rigid relation, we get a nontrivial endomorphism $h:(G,E)\rightarrow(G,E).$ Again, $h$ must fix every vertex in $H,$ and the two sets $K^{\mathrm{e}}$ and $K^{\mathrm{o}}$ are each closed under $h$ by G1.

    We want to prove that if a member of some class $[x]$ is sent by $h$ to a member of some class $[y],$ then every member of $[x]$ is sent to a member of $[y].$ We will show this for the classes in $\subbar{N},$ the argument for the classes in $\subbar{N^c}$ is done similarly. To do this, we will prove that $h((r_1,x))\in [y_1]\ \land\ h((r_2,x))\in[y_2]\implies y_1=y_2,$ for every $x,y_1,y_2\in N$ and every $(r_1,x),(r_2,x)\in[x].$ Suppose, towards a contradiction, that $h((r_1,x))\in [y_1]$ and $h((r_2,x))\in [y_2],$ but $y_1\neq y_2.$ We know that there exists some real number $s$ such that $(s,x^c)\in K^{\mathrm{o}},$ because $\pi_2[K^{\mathrm{o}}]=N^c\ni x^c.$ By G1, $((r_1,x),(s,x^c))\in E.$ Applying $h$ to this, we get $(h((r_1,x)),h((s,x^c)))\in E.$ By closure of $K^{\mathrm{e}}$ and $K^{\mathrm{o}},$ $h((r_1,x))\in K^{\mathrm{e}}$ and $h((s,x^c))\in K^{\mathrm{o}}.$ Therefore, by G1, $\pi_2(h((s,x^c)))=\pi_2(h((r_1,x)))^c=y_1^c.$ But, by symmetry, we also get $\pi_2(h((s,x^c)))=\pi_2(h((r_2,x)))^c=y_2^c.$ This is a contradiction since $y_1\neq y_2.$

    Using the above, we can define a function $j:N\cup N^c\rightarrow N\cup N^c,$ by letting $j(x)=y$ iff $h((r,x))=(s,y)$ for some $(r,x)\in [x]$ and $(s,y)\in [y].$ Now, by dropping the notations $\subbar{-}$ and $[-],$ we can argue for elementarity of $j\vert_{V_\alpha^L}:V_\alpha^L\rightarrow V_\alpha^L$ just as we argued for elementarity of $h\vert_{V_\alpha^L}$ in the proof of Theorem~\ref{thm:ordinal_case}. There will be just one detail that will be different, and that is, for the injectivity argument, we will need the fact that there are no arrows between vertices belonging to the same class. Using that, we can argue that the members of two distinct $[x^c],[y^c]$ cannot be sent by $h$ to the same class $[z^c].$

    Observe that, although we have established the elementarity of $j\vert_{V_\alpha^L},$ the argument in Theorem~\ref{thm:ordinal_case} is not sufficient to prove that $j\vert_{V_\alpha^L}$ is nontrivial. What could happen is that $j=\id,$ but the nontriviality of $h$ happens inside some class $[x],$ where $h((r_1,x))=(r_2,x)$ for some distinct $(r_1,x),(r_2,x)\in [x].$ Condition G6 is there precisely to prohibit this scenario. The argument is the same as in the proof of Claim~\ref{claim:thm_reals_case} with its G2.
\end{proof}

In the context of $\ZF,$ a cardinal $\kappa$ is said to be \emph{inaccessible} iff there is no function $f:V_\alpha\rightarrow \kappa$ with cofinal range, for any $\alpha<\kappa.$ Just as in the context of $\ZFC,$ it is easy to show that $(V_\kappa,V_{\kappa+1})\models \ZF_2$ whenever $\kappa$ is inaccessible. 
Since the critical point of any nontrivial elementary embedding $j:V_\alpha^L\rightarrow V_\alpha^L$ is an inaccessible cardinal in $L,$ the following corollary is an immediate consequence of the previous theorem.

\begin{corollary}\label{cor:no_inaccessibles_to_srr}
    If there are no inaccessible cardinals in $L,$ then every set $G\subset \mathbb{R}\times\OR$ has a strongly rigid relation.
\end{corollary}


\section{A Model for $\neg \AC+\SRR$}

We are now ready to build a model for $\ZF+\neg\AC+\SRR,$ thus establishing that $\SRR$ does not imply $\AC.$

\begin{theorem}\label{thm:neg_ac_plus_srr}
    If $\ZF$ is consistent, then so is $\ZF+\neg\AC+\SRR.$
\end{theorem}

\begin{proof}
    Work inside $V.$ We can assume that there are no inaccessible cardinals in $L$ (if there are any, then simply work inside $V_\kappa^L$ for $\kappa$ the least inaccessible in $L).$ Given this assumption and Corollary~\ref{cor:no_inaccessibles_to_srr}, the same model that worked for Hamkins and Palumbo \cite{hamkins} for the consistency of $\ZF+\neg\AC+\RR$ will work for us too. The model in question is the symmetric Cohen model $M,$ built as follows: Let $P=\operatorname{Add}(\omega,\omega)$ be the usual forcing notion that adds countably many Cohen reals. Thus, $P=\{p:\omega\times\omega \rightarrow\{0,1\}\mid p\textrm{ is finite}\}$ and $p\leq q\iff p\supset q,$ for all $p,q\in P.$ Every permutation $\pi:\omega\rightarrow \omega$ induces an automorphism $\bar\pi:P\rightarrow P$ by letting
    $$
    \bar\pi(p)=\{(\pi(n),m)\mid (n,m)\in p\}.
    $$
    This automorphism, in turn, induces an automorphism $\hat\pi:V^P\rightarrow V^P$ of the class of $P$-names by the recursive definition
    $$
    \hat\pi(\tau)=\{(\hat\pi(\sigma),\bar\pi(p))\mid(\sigma,p)\in\tau\}.
    $$
    It is easy to prove, using induction on complexity of formula, that $p\forces_P \psi[\tau] \iff \bar\pi(p)\forces_P \psi[\hat\pi(\tau)].$
    A $P$-name $\tau$ is said to be \emph{symmetric} iff there exists a finite set $e\subset \omega$ such that whenever $\pi:\omega\rightarrow \omega$ is a permutation that fixes every member of $e,$ then $\hat\pi(\tau)=\tau.$ The class of hereditarily symmetric names is denoted by $\mathrm{HS}.$ The symmetric Cohen model $M$ is defined in an extension $V[G]$ to be the class of the interpretations of hereditarily symmetric names, that is, $M=\{i_G(\tau)\mid \tau \in \mathrm{HS}\}.$

    In \cite[Lemma~5.25]{jech_ac_book}, it is established in $M$ that there is a set of reals $A$ such that every set can be injected into $A^{<\omega}\times\OR.$ Working in $M,$ as $\mathbb{R}$ and $\mathbb{R}^{<\omega}$ are in bijection, this means that every set is in bijection with a subset of $\mathbb{R}\times \OR.$ But, since $L^M=L^V,$ the assumption that there are no inaccessible cardinals in $L^M$ and Corollary~\ref{cor:no_inaccessibles_to_srr} imply that every subset of $\mathbb{R}\times \OR,$ and hence every set, has a strongly rigid relation.
\end{proof}


\section{Proto Berkeley Cardinals} 

\begin{theorem}\label{thm:proto_berkeley_cardinals}
    A cardinal $\delta$ is proto Berkeley iff for any graph $(G,E)$ such that $\aleph(G)>\delta$ and any injection $f:\delta\rightarrow G,$ there is an endomorphism $h:(G,E)\rightarrow(G,E)$ such that $h\vert_{\range{f}}\neq \id.$
\end{theorem}

\begin{proof}
    From left to right is easy, so we do that first. Fix any $(G,E)$ and $f:\delta\rightarrow G$ as in the statement of the theorem. Let $V_\alpha$ be such that $\langle G,E,f\rangle \in V_\alpha,$ and let $M=V_\alpha\cup\{V_\alpha,\{\langle G,E,f\rangle,V_\alpha\}\}.$ Clearly, $M$ is a transitive set and $\delta\in M.$ By proto Berkeleyness, we can find an elementary embedding $j:M\rightarrow M$ that has critical point $\kappa$ strictly below $\delta.$ The pair $\{\langle G,E,f\rangle,V_\alpha\}$ is definable in $M$ as the unique set with exactly two members, one of which is a set that has every other set as a member. It easily follows from this that each of $G,E,f$ is definable in $M,$ and we deduce that $j$ must fix each one of them. This shows that $j\vert_G:(G,E)\rightarrow(G,E)$ is an endomorphism. Also, $j(f(\kappa))=j(f)(j(\kappa))=f(j(\kappa))\neq f(\kappa),$ so that $j\vert_G$ is not the identity on the range of $f.$

    For the other direction, fix a transitive set $M$ such that $\delta\in M.$ Our aim is to show that there is a nontrivial elementary embedding $j:M\rightarrow M$ with critical point strictly below $\delta.$ The idea of the proof is similar to that of the proof of Theorem~\ref{thm:ordinal_case}. Thus, let $N\supset M\cup\{M\}$ be the smallest transitive set closed under taking finite subsets, and let $N^c$ be a copy of it.
    
    (We need a countable set disjoint from $N\cup N^c$ that will be fixed by any endomorphism to play the role of the set $A\cup B$ of Theorem~\ref{thm:ordinal_case}, and we also need a relation on $N^c$ which can be used to argue for injectivity of any endomorphism on $N.$ For the injectivity argument, we need to have an arrow between any two distinct $x^c,y^c\in N^c.$ For this, we made use of $<_L$ in Theorem~\ref{thm:ordinal_case}, but in our current situation, there is no guarantee that $N$ is wellorderable. An obvious candidate for our case is $\neq,$ but this causes another problem: In the proof of Theorem~\ref{thm:ordinal_case}, when we argued that $A$ was fixed, we needed the fact that there were no 3-cycles in the graph other than the one in $A.$ However, if we use $\neq,$ we will be adding cycles of all sizes to the graph, and the argument for $A$ being fixed fails. One remedy to this is to replace $A$ with another solution. That is where we will use the notion of Hartog's number.)
    
    Let $A=\alpha+1,$ where $\alpha=\aleph(N\cup N^c).$ Consider the graph $(G,E)$ in Figure~\ref{fig:berkeley_case}, where
    \begin{enumerate}[label=G\arabic*]
        \item $(w_0,x),(x,x^c),(w_1,x^c)\in E,$ for every $x\in N.$
        \item $(x,y)\in E$ iff $x\in y,$ for every $x,y\in N.$
        \item $(x^c,y^c)\in E$ iff $x^c\neq y^c,$ for every $x^c,y^c\in N^c.$
        \item $(p_n,x^c)\in E$ iff $|x|=n,$ for every $n\in\omega$ and every $x^c\in N^c.$
        \item For every $n\in \omega$ and every $x^c\in N^c,$ $(q_n,x^c)\in E$ iff one of the followings hold:
        \begin{enumerate}[label=(\alph*)]
            \item $x\in M,$ $n=\ulcorner\psi(a)\urcorner$ for some $\psi(a)\in\FORM,$ and $M\satisfies \psi[x].$
            \item $x=\langle x_1,\ldots,x_m\rangle$ for some $m>1,$ $x_1,\ldots,x_m\in M,$ $n=\ulcorner \psi(a_1,\ldots,a_m)\urcorner$ for some $\psi(a_1,\ldots,a_m)\in \FORM,$ and $M\satisfies \psi[x_1,\ldots,x_m].$
        \end{enumerate}
        \item $(\xi_1,\xi_2)\in E$ iff $\xi_1<\xi_2,$ for every $\xi_1,\xi_2\in \alpha+1.$
    \end{enumerate}
    
    \begin{figure}[htb!]
        \centering
        \includegraphics[page=4]{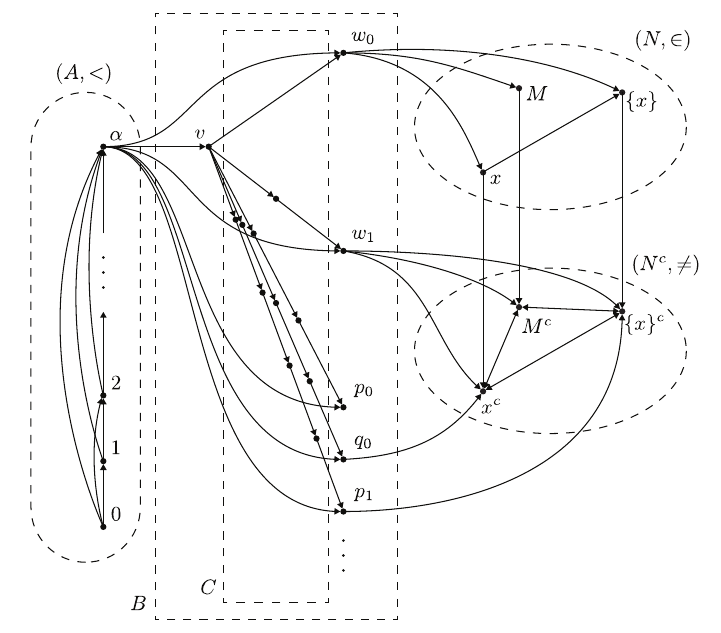}
        \caption{}
        \label{fig:berkeley_case}
    \end{figure}

    We remark that although $A$ and $N$ have some common ordinals, we really mean that they are disjoint and merely use their real names for convenience. Normally one would use a copy of, say, the set $A,$ just as we did with $N.$ We will aim to make it clear from the context which copy of a given ordinal we are talking about.

    Since $\aleph(G)>\delta,$ we can fix an endomorphism $h:(G,E)\rightarrow (G,E)$ that is nontrivial on $\delta\subset N,$ if we take as $f$ the injection $\id:\delta\rightarrow \delta\subset N.$ That will take care of the critical point being strictly below $\delta\in M,$ and we only need to show that $h\vert_M:M\rightarrow M$ is an elementary embedding. Although we cannot argue that every member of $A$ is fixed by $h,$ we can, nonetheless, argue that $\alpha\in A$ is fixed by $h.$ Using that, we can proceed as in the proof of Theorem~\ref{thm:ordinal_case} and show that $h\vert_M:M\rightarrow M$ is indeed an elementary embedding, which will finish the proof.
    
    So, let us show that $\alpha$ is fixed by $h.$ Henceforth, every ordinal we name belongs to $A,$ and not $N.$ First, we argue that $h(\beta)\in A,$ for all $\beta<\alpha.$ Working towards a contradiction, fix $\beta<\alpha$ such that $h(\beta)\notin A.$ If $h(\gamma)\in A$ for some $\gamma>\beta,$ then the arrow $(h(\beta),h(\gamma))\in E$ will be flowing from outside of $A$ into $A.$ But, no such arrows exist in our graph, so we must conclude that $h(\gamma)\notin A,$ for all $\gamma>\beta.$ We have established that $h\vert_{[\beta,\alpha]}$ has $G-A$ as codomain. By choice of $\alpha,$ $h\vert_{[\beta,\alpha]}$ cannot be an injection. Therefore, there are $\beta\leq\xi_1<\xi_2<\alpha$ such that $h(\xi_1)=h(\xi_2).$ This is a contradiction, since $(\xi_1,\xi_2)\in E$ implies $(h(\xi_1),h(\xi_2))\in E,$ which is a loop, and our graph is free of loops.

    Fix any $\beta_1<\beta_2<\alpha.$ By the above and endomorphism of $h,$ $h(\beta_1)<h(\beta_2)\leq \alpha.$ Since $(h(\beta_1),h(\alpha))\in E$ and there are no arrows flowing from $A-\{\alpha\}$ to outside of $A,$ we must have $h(\alpha)\in A$ too. Hence, either $h(\alpha)=\alpha$ or $h(\alpha)<\alpha.$ Suppose, again towards a contradiction, that $h(\alpha)=\beta<\alpha.$ Taking any $\gamma<\alpha,$ we see that $h(\gamma)<h(\alpha)=\beta.$ Hence, $h\vert_{[0,\alpha)}$ has $[0,\beta)$ as codomain. Since $\alpha$ is clearly a cardinal and $\beta<\alpha,$ $h\vert_{[0,\alpha)}$ cannot be injective. Again, this implies the existence of a loop, which is a contradiction. The only option we are left with is $h(\alpha)=\alpha.$
\end{proof}


There is another way to deal with the problem of the injectivity argument in the proof above that does not involve creating cycles. For this, we need two copies of $N,$ call them $N^{c_1}$ and $N^{c_2}.$ Define $<_\rk$ by $x<_\rk y \iff \rank(x)<\rank(y),$ for all sets $x,y.$ Define also $\notin_\rk$ by $x\notin_\rk y \iff x\notin y \ \land\ x<_\rk y,$ for all sets $x,y.$ Let $<_\rk^{c_2}\subset N^{c_2}\times N^{c_2}$ be defined by $x^{c_2}<_\rk^{c_2}y^{c_2} \iff x<_\rk y,$ for all $x^{c_2},y^{c_2}\in N^{c_2}.$ Similarly, let $\notin_\rk^{c_1}\subset N^{c_1}\times N^{c_1}$ be defined by $x^{c_1}\notin_\rk^{c_1} y^{c_1}\iff x\notin_\rk y,$ for all $x^{c_1},y^{c_1}\in N^{c_1}.$

Consider the graph $(G,E)$ in Figure~\ref{fig:another_solution_for_injectivity}, where
\begin{enumerate}[label=G\arabic*]
    \item $(x,x^{c_1}),(x^{c_1},x^{c_2})\in E,$ for every $x\in N$
    \item $(x,y)\in E$ iff $x\in y,$ for every $x,y\in N$
    \item $(x^{c_1},y^{c_1})\in E$ iff $x^{c_1} \notin_\rk^{c_1} y^{c_1},$ for every $x^{c_1},y^{c_1}\in N^{c_1}$
    \item $(x^{c_2},y^{c_2})\in E$ iff $x^{c_2}<_\rk^{c_2} y^{c_2},$ for every $x^{c_2},y^{c_2}\in N^{c_2}.$
\end{enumerate}

\begin{figure}[htb!]
    \centering
    \includegraphics[page=8]{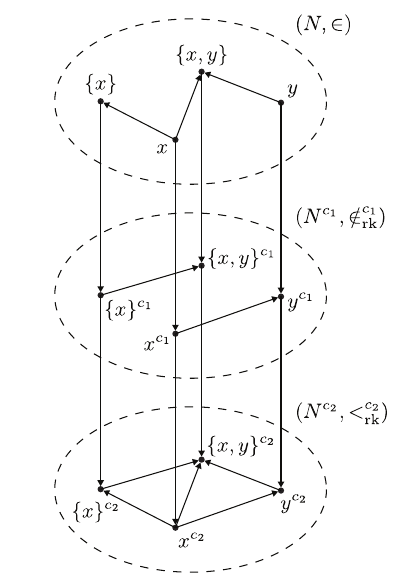}
    \caption{}
    \label{fig:another_solution_for_injectivity}
\end{figure}

Clearly, this graph does not contain any cycles by wellfoundedness of $\in$ and $<_\rk.$ We can use this graph as part of a bigger graph which ensures that each of the copies of $N$ are closed under any endomorphism $h,$ and which also ensures elementarity of the appropriate restrictions of $h.$ Here, we will only show:

\begin{proposition}
    If $(G,E)$ is the graph in Figure~\ref{fig:another_solution_for_injectivity} and $h:(G,E)\rightarrow (G,E)$ is any endomorphism, under which each of the copies of $N$ are closed, then $h$ is injective on $N$ and behaves the same across all copies of $N.$
\end{proposition}

\begin{proof}
    Fix such an endomorphism $h.$ By G1, 
    \begin{enumerate}[label=C\arabic*]
        \item $h(x^{c_1})=h(x)^{c_1},$ for every $x\in N$
        \item $h(x^{c_2})=h(x)^{c_2},$ for every $x\in N.$
    \end{enumerate}
    This shows that $h$ behaves the same across all copies of $N.$ Now, for injectivity,
    fix any $x,y\in N$ such that $x\neq y.$ If $x<_\rk y$ or $y<_\rk x,$ then $h(x)<_\rk h(y)$ or $h(y)<_\rk h(x)$ by G4, applying $h,$ and C2. Thus, in such cases, $h(x)\neq h(y).$ Suppose now that $\rank(x)=\rank(y).$ Fix $a\in N$ such that, without loss of generality, $a\in x$ but $a\notin y.$ $a\in x$ implies $\rank(a)<\rank(x)=\rank(y).$ That, along with $a\notin y,$ imply that $a\notin_\rk y.$ By G3, applying $h,$ and C1, we get $h(a)\notin_\rk h(y).$ In particular, $h(a)\notin h(y).$ Meanwhile, we also have $h(a)\in h(x)$ by G2 and applying $h.$ That is, $h(a)\in h(x)$ but $h(a)\notin h(y),$ which is to say $h(x)\neq h(y).$ As $x,y$ were arbitrary, we deduce that $h$ is injective on $N.$
\end{proof}




\printbibliography

\end{document}